\documentclass[11pt,reqno]{amsart}

\usepackage{amsmath, amstext, amsfonts}
\usepackage[left=3cm,top=2.5cm,right=3cm,bottom=2.5cm]{geometry}
\usepackage[colorlinks,bookmarks]{hyperref}

\newtheorem{thm}{Theorem}

\DeclareMathOperator{\diam}{diam}

\DeclareMathOperator*{\DSpec}{DSpec}

\title{On the minimum number of eigenvalues of matrices associated with cographs}
\author[Allem]{Luiz Emilio Allem}
\address{Instituto de Matem\'{a}tica e Estat\'{i}stica, Universidade Federal do Rio Grande do Sul, Brazil}
\email{emilio.allem@ufrgs.br}

\author[F\"{u}rer]{Martin F\"{u}rer}
\address{Department of Computer Science and Engineering, Pennsylvania State University, USA}  
\email{fhs@psu.edu}

\author[Hoppen]{Carlos Hoppen}
\address{Instituto de Matem\'atica e Estat\'{i}stica, Universidade Federal do Rio Grande do Sul, Brazil}
\email{choppen@ufrgs.br}

\author[Sibemberg]{Lucas Siviero Sibemberg}
\address{Instituto de Matem\'atica e Estat\'{i}stica, Universidade Federal do Rio Grande do Sul, Brazil}
\email{lucas.siviero@ufrgs.br}

\author[Trevisan]{Vilmar Trevisan}
\address{Instituto de Matem\'{a}tica e Estat\'{i}stica, Universidade Federal do Rio Grande do Sul, Brazil}
\email{trevisan@mat.ufrgs.br}

\begin{document}

\begin{abstract}
A symmetric matrix $M=(m_{ij}) \in \mathbb{R}^{n \times n}$ is said to be associated with an $n$-vertex graph $G=(V,E)$ with vertex set $\{v_1,\ldots,v_n\}$ if, for every $i \neq j$, we have $m_{ij} \neq 0$ if and only if $\{v_i,v_j\}\in E$. We prove that, for every cograph $G$, there is a matrix $M$ associated with $G$ for which the number of distinct eigenvalues is at most 4.
\end{abstract}

\maketitle

\section{Introduction}

Let $G$ be a simple graph with vertex set $V=\{1,2,\ldots,n\}$ and edge set $E$. We associate a collection $S(G)$ of real $n\times n$ symmetric matrices with $G$:
$$S(G)=\left\{M \in \mathbb{R}^{n \times n} \colon M=M^{T} \mbox{ and } \left(\forall~i\neq j \right) \left(a_{ij}\neq 0 \Longleftrightarrow \{i,j\}\in E\right)\right\}.$$
This means that the off-diagonal nonzero entries of a matrix $M \in S(G)$ are precisely the entries $ij$ for which there is an edge $\{i,j\}\in E$, while there is no constraint on entries on the main diagonal. This short note is concerned with the parameter $q(G)$, the minimum number of distinct eigenvalues in a matrix in $S(G)$.  Formally, if $\DSpec(M)$ denotes the minimum number of eigenvalues in a matrix $M$, we have
$$q(G)=\mbox{min}\{|\DSpec(M)|\colon M\in S(G)\}.$$

For general results about $q(G)$, we refer to Hogben, Lin and Shader~\cite{hogben2022inverse}.

The parameter $q(G)$ has been investigated extensively for several classes of graphs. The interest in the spectrum for matrices associated with trees stands out. We should mention early work of Parter~\cite{Parter60} and Wiener~\cite{WIENER1984}, and the systematic study initiated by Leal Duarte, Johnson and their collaborators, see for example~\cite{ParterWiener03,JOHNSON20027,DUARTE1989173,leal2002minimum,2018eigenvalues}. As it turns out, for any tree $T$, $q(T) \geq \diam(T)+1$, where $\diam(T)$ is the diameter of the tree, that is, the number of edges on a longest path in $T$. There are instances for which this inequality is strict, but equality holds for large families of trees, which are therefore known as diminimal trees. For instance, Leal Duarte and Johnson~\cite{leal2002minimum} have shown that all trees of diameter at most 5 are diminimal. Moreover, a tree $T$ is linear if all vertices of degree at least three lie on a common path. Johnson and Wakhare~\cite{johnson2022inverse} have shown that all linear trees are diminimal. Allem et al~\cite{allem2023diminimal} described infinite families of diminimal trees of any fixed diameter, many of which are not linear. 

The current paper is concerned with \emph{complement reducible graphs}, commonly known as \emph{cographs}. These are the graphs that do not contain the 4-vertex path $P_4$ as an induced subgraph, which have been considered in various contexts and have been rediscovered multiple times. For a historical account and the terminology associated with cographs, we refer the reader to a  seminal paper of Corneil, Lerchs, and Burlingham~\cite{Corneil1981}. Recently, Lazzarin, Tura, and three of the current authors~\cite{ALLEM202532} have considered threshold graphs, the subclass of cographs that contains all graphs with no induced copy of $P_4$, $C_4$ and $2K_2$. They proved the following result:
\begin{thm}\label{main_theorem_threshold}
If $G$ is a threshold graph and $\lambda \neq 0$ is a real number, then there is a matrix $M\in S(G)$ such that $\DSpec(M)\subseteq\{-\lambda,0,\lambda,2\lambda\}$. In particular, $q(G) \leq 4$.
\end{thm}
As it turns out, the upper bound is tight, in the sense that there exists a family of threshold graphs such that all $G$ in this family satisfy $q(G)=4$, see~\cite[Theorem 6.4]{fallat2022minimum}.

In this note, we show that the same conclusion still holds for the much larger class of cographs. 
\begin{thm}\label{main:thm}
If $G$ is a cograph and $\lambda \neq 0$ is a real number, then there is a matrix $M\in S(G)$ such that $\DSpec(M)\subseteq\{-\lambda,0,\lambda,2\lambda\}$. 
\end{thm}

\section{Cographs}

Cographs may be defined in several different ways, one of them being through \emph{(disjoint) unions} and \emph{joins}. Given graphs $G_1=(V_1,E_1)$ and $G_2=(V_2,E_2)$ whose vertex sets are disjoint, the union $G_1 \cup G_2$ is the graph with vertex set $V=V_1 \cup V_2$ and edge set $E=E_1 \cup E_2$.  This can be naturally extended to the union $G_1 \cup\cdots \cup G_k$ of a finite family of mutually vertex-disjoint graphs $G_1=(V_1,E_1),\ldots,G_k=(V_k,E_k)$. The join $G_1 \oplus G_2$ is the graph with vertex set $V=V_1 \cup V_2$ and edge set $E=E_1 \cup E_2 \cup\{\{v_i,v_j\} \colon v_i \in V_1,~v_j \in V_2\}$, which can again be naturally extended to the join $G_1 \oplus \cdots \oplus G_k$ of a finite family of mutually vertex-disjoint graphs. With these operations, one may define the class $\mathcal{C}$ of cographs in a recursive way:
\begin{itemize}
\item[(a)]~$K_1 \in \mathcal{C}$.

\item[(b)]~If $G_1$ and $G_2$ lie in $\mathcal{D}$, then $G_1 \cup G_2 \in \mathcal{C}$.

\item[(c)]~If $G_1$ and $G_2$ lie in $\mathcal{D}$, then $G_1 \oplus G_2 \in \mathcal{C}$.
\end{itemize}
This recursive definition allows one to represent every cograph with vertex set $[n]$ as a rooted tree whose \emph{nodes} consist of $n$ leaves labeled $1$ through $n$ and internal vertices that carry either the label ``$\cup$'' for union or ``$\oplus$'' for join. Given such a tree $T$, we construct the corresponding cograph $G_T$ as follows: arbitrarily order the nodes of $T$ as $1,\ldots,m$ bottom-up. For $i$ from 1 to $m$, process node $i$ as follows. If it is a leaf labeled $j$, produce a singleton whose vertex is labeled $j$. If it is an internal node labeled $\cup$, take the union of the graphs produced by its children. If it is an internal node labeled $\oplus$, take the join of the graphs produced by its children. A tree $T$ associated with a graph $G=G_T$ will be called a \emph{cotree} of $G$. A cotree contains a lot of information about the corresponding cograph. For instance, fix distinct vertices $u,v \in V(G)$ and let $x$ be the lowest common ancestor of the leaves corresponding to $u$ and $v$ in $T$. If $x$ has label $\oplus$, then $u$ and $v$ are adjacent; if $x$ has label $\cup$, then they are not. In particular, the cotree $\overline{T}$ obtained from $T$ by switching labels $\oplus$ and $\cup$ produces the complement of $G$, that is, $G_{\overline{T}}=\overline{G}$.

We say that a cotree is in \emph{normalized form} if every internal node has at least two children and has a label that differs from the label of its parent. In other words, the children of nodes labeled $\cup$ are leaves or nodes labeled $\oplus$, while the children of nodes labeled $\oplus$ are leaves or nodes labeled $\cup$. If the root of such a tree $T$ is labeled $\cup$, then $G_T$ is disconnected; otherwise, $G_T$ is connected. While the same graph may be generated by different cotrees, there is a single cotree in normalized form associated with each graph.

Another characterization of cographs may be obtained as follows. Two vertices $u$ and $v$ in a graph $G$ are {\em false twins} (or \emph{duplicates}) if their neighborhoods $N(u)$ and $N(v)$ are the same, that is, if they are adjacent to exactly the same vertices. The vertices $u$ and $v$ are {\em true twins} (or \emph{coduplicates}) if they are adjacent and they are duplicates in the subgraph of $G$ obtained by deleting the edge $\{u,v\}$. It turns out that $u$ and $v$ of a cograph $G$ are twins if and only if they share the same parent node $x$ in the cotree $T$ in normalized form associated with $G$. Moreover, if $x = \cup$, they are false twins. If $x = \oplus$ they are true twins. From this, we may easily prove that cographs are precisely the graphs that may be constructed from a singleton by a sequence of duplications and co-duplications.
\begin{thm}\label{thm:basic}
An $n$-vertex graph $G$ is a cograph if and only if there is a graph sequence $G_1,G_2,\ldots, G_n$ such that $G_1$ has order 1, $G_n=G$ and, for $i \in \{1,\ldots,n-1\}$, $G_{i+1}$ is obtained from $G_i$ by the addition of a twin of a vertex $u \in V(G_i)$.
\end{thm}

\section{The minimum number of eigenvalues of cographs}

In this section, we prove that $q(G) \leq 4$ for every cograph $G$. This may be done by induction on the number $n$ of vertices of $G$. We shall prove a slightly stronger result.
\begin{thm}\label{thm:main}
Given a real number $\lambda \neq 0$ and a cograph $G$, there is a matrix $M \in \mathcal{S}(G)$ such that $q(M) \subseteq \{-\lambda,0,\lambda,2\lambda\}$ with the additional property that every element in the diagonal of $M$ belongs to $\{0,\lambda\}$.
\end{thm}

\begin{proof}
Let $G=(V,E)$ be an $n$-vertex cograph. Our proof proceeds by induction on the number of vertices. If $n=1$, the matrix $M=[0]$ clearly satisfies the required conditions. So, assume that $n>1$ and that the result holds for cographs with fewer vertices.

By Theorem~\ref{thm:basic}, $G$ contains a pair of twins $v,v'$. Let $G^\ast=G-v'$ be the cograph obtained by removing vertex $v'$, and let $M^\ast=(m^\ast_{ij})$ be a matrix associated with $G^\ast$ that satisfies the conditions of Theorem~\ref{thm:main}. Let $x_1^\ast,\ldots,x_{n-1}^\ast$ be an orthonormal basis of $\mathbb{R}^{n-1}$ given by eigenvectors on $M^\ast$. Our goal is to define an orthonormal basis $x_1,\ldots,x_{n}$ of $\mathbb{R}^n$ whose elements are eigenvectors associated with eigenvalues in  $\{-\lambda,0,\lambda,2\lambda\}$ with respect to a matrix $M \in \mathcal{S}(G)$. Without loss of generality, assume that the vertex set $v_1,\ldots,v_n$ of $G$ is ordered so that $v_{n-1}=v$ and $v_n=v'$. 

As it turns out, using the notation $x_i(u)$ to denote the component of $x_i$ associated with a vertex $u$, we may consider the following vectors in $\mathbb{R}^n$:
\begin{eqnarray}\label{vec:x}
&&x_1=\left[ 
\begin{array}{c}
x_1^\ast(v_1)\\
\vdots\\
x_1^\ast(v_{n-2})\\
x_1^\ast(v)/\sqrt{2}\\
x_1^\ast(v)/\sqrt{2}
\end{array}\right],\ldots,~x_{n-1}=\left[ 
\begin{array}{c}
x_{n-1}^\ast(v_1)\\
\vdots\\
x_{n-1}^\ast(v_{n-2})\\
x_{n-1}^\ast(v)/\sqrt{2}\\
x_{n-1}^\ast(v)/\sqrt{2}
\end{array}\right],
~x_n=\left[ 
\begin{array}{c}
0\\
\vdots\\
0\\
1/\sqrt{2}\\
-1/\sqrt{2}
\end{array}\right].
\end{eqnarray}
The fact that this is an orthonormal basis of $\mathbb{R}^n$ follows immediately from our assumption $x_1^\ast,\ldots,x_{n-1}^\ast$. 

To construct $M=(m_{uw})$, we consider four cases.

\vspace{5pt}

\noindent \textbf{Case 1:} $v$ and $v'$ are false twins and $m^*_{vv}=0$

\vspace{5pt}

Define $M=(m_{uw})$ as follows:
\begin{equation}\label{case1}
\begin{array}{ll}
m_{uw}=m^*_{uw}, & \textrm{ if }\{u,w\} \cap \{v,v'\}=\emptyset,\\
m_{uv}=m_{vu}=m^*_{uv}/\sqrt{2}, & \textrm{ if }u \notin \{v,v'\},\\
m_{uv'}=m_{v'u}=m^*_{uv}/\sqrt{2}, & \textrm{ if }u \notin \{v,v'\},\\
m_{vv}=m_{v'v'}=m_{vv'}=m_{v'v}=0.&
\end{array}
\end{equation}

Suppose that the eigenvector $x^*_i$, where $i \in [n-1]$, is associated with the eigenvalue $\alpha$ of $M^*$. For $u \notin \{v,v'\}$, we have
\begin{eqnarray*}
(Mx_i)_u&=&\sum_{w \in V(G)} m_{uw} x_i(w) = m_{uv} x_i(v)+ m_{uv'} x_i(v')+\sum_{w \notin \{v,v'\}} m_{uw} x_i(w)\\
&=& \frac{m^*_{uv}}{\sqrt{2}} \cdot \frac{x^*_i(v)}{\sqrt{2}} +\frac{m^*_{uv}}{\sqrt{2}} \cdot \frac{x^*_i(v)}{\sqrt{2}} + \sum_{w \notin \{v,v'\}} m^*_{uw} x^*_i(w)\\
&=& \sum_{w \in V(G^*)} m^*_{uw} x^*_i(w) = \alpha x^*_i(u) = \alpha x_i(u).
\end{eqnarray*}
We also have
\begin{eqnarray*}
(Mx_i)_v&=&\sum_{w \in V(G)} m_{vw} x_i(w) = m_{vv} x_i(v)+ m_{vv'} x_i(v')+\sum_{w \notin \{v,v'\}} m_{vw} x_i(w)\\
&=& 0+0+\sum_{w \notin \{v,v'\}} \frac{m^*_{vw}}{\sqrt{2}} \cdot x^*_i(w)=0 \cdot \frac{x^*_i(v)}{\sqrt{2}}+ \sum_{w \notin \{v,v'\}} \frac{m^*_{vw}}{\sqrt{2}} \cdot x^*_i(w)\\
&=& \frac{1}{\sqrt{2}} \sum_{w \in V(G^*)} m^*_{vw} x^*_i(w) =  \frac{1}{\sqrt{2}} \cdot ( \alpha x^*_i(v) ) = \alpha x_i(v).
\end{eqnarray*}
Similarly, $(Mx_i)_{v'}=\alpha x_i(v')$, so that $x_i$ is an eigenvector of $M$ associated with $\alpha \in \{-\lambda,0,\lambda,2\lambda\}$.

Regarding the vector $x_n$, for $u \neq \{v,v'\}$, we have
\begin{eqnarray*}
(Mx_n)_u&=&\sum_{w \in V(G)} m_{uw} x_n(w) = m_{uv} x_n(v)+ m_{uv'} x_n(v')+\sum_{w \notin \{v,v'\}} m_{uw} x_n(w)\\
&=& \frac{m^*_{uv}}{\sqrt{2}} \cdot \frac{1}{\sqrt{2}} +\frac{m^*_{uv}}{\sqrt{2}} \cdot \frac{-1}{\sqrt{2}} + \sum_{w \notin \{v,v'\}} m^*_{uw} \cdot 0=0.
\end{eqnarray*}
We also have
\begin{eqnarray*}
(Mx_n)_v&=&\sum_{w \in V(G)} m_{vw} x_n(w) = m_{vv} x_n(v)+ m_{vv'} x_n(v')+\sum_{w \notin \{v,v'\}} m_{vw} x_n(w)\\
&=& 0+0+ \sum_{w \notin \{v,v'\}} \frac{m^*_{vw}}{\sqrt{2}} \cdot 0=0.
\end{eqnarray*}
Similarly, $(Mx_n)_{v'}=0$, so that $x_n$ is an eigenvector of $M$ associated with the eigenvalue 0. 

As a consequence, $x_1,\ldots,x_{n}$ are indeed a basis of $\mathbb{R}^n$ consisting of eigenvalues of $M$. Any eigenvalue of $M$ is equal to an eigenvalue of $M^*$ or to 0, so it lies in $\{-\lambda,0,\lambda,2\lambda\}$, as required.

\vspace{5pt}

\noindent \textbf{Case 2:} $v$ and $v'$ are false twins and $m^*_{vv}=\lambda$

\vspace{5pt}

Define $M=(m_{uw})$ as follows:
\begin{equation}\label{case2}
\begin{array}{ll}
m_{uw}=m^*_{uw}, & \textrm{ if }\{u,w\} \cap \{v,v'\}=\emptyset,\\
m_{uv}=m_{vu}=m^*_{uv}/\sqrt{2}, & \textrm{ if }u \notin \{v,v'\},\\
m_{uv'}=m_{v'u}=m^*_{uv}/\sqrt{2}, & \textrm{ if }u \notin \{v,v'\},\\
m_{vv}=m_{v'v'}=\lambda,m_{v'v}=m_{vv'}=0.&
\end{array}
\end{equation}

Suppose that the eigenvector $x^*_i$, where $i \in [n-1]$, is associated with the eigenvalue $\alpha$ of $M^*$. For $u \notin \{v,v'\}$, we have
\begin{eqnarray*}
(Mx_i)_u&=&\sum_{w \in V(G)} m_{uw} x_i(w) = m_{uv} x_i(v)+ m_{uv'} x_i(v')+\sum_{w \notin \{v,v'\}} m_{uw} x_i(w)\\
&=& \frac{m^*_{uv}}{\sqrt{2}} \cdot \frac{x^*_i(v)}{\sqrt{2}} +\frac{m^*_{uv}}{\sqrt{2}} \cdot \frac{x^*_i(v)}{\sqrt{2}} + \sum_{w \notin \{v,v'\}} m^*_{uw} x^*_i(w)\\
&=& \sum_{w \in V(G^*)} m^*_{uw} x^*_i(w) = \alpha x^*_i(u) = \alpha x_i(u).
\end{eqnarray*}
We also have
\begin{eqnarray*}
(Mx_i)_v&=&\sum_{w \in V(G)} m_{vw} x_i(w) = m_{vv} x_i(v)+ m_{vv'} x_i(v')+\sum_{w \notin \{v,v'\}} m_{vw} x_i(w)\\
&=& \lambda \cdot \frac{x^*_i(v)}{\sqrt{2}}+0+\sum_{w \notin \{v,v'\}} \frac{m^*_{vw}}{\sqrt{2}} \cdot x^*_i(w)=m^*_{vv} \cdot \frac{x^*_i(v)}{\sqrt{2}}+ \sum_{w \notin \{v,v'\}} \frac{m^*_{vw}}{\sqrt{2}} \cdot x^*_i(w)\\
&=& \frac{1}{\sqrt{2}} \sum_{w \in V(G^*)} m^*_{vw} x^*_i(w) =  \frac{1}{\sqrt{2}} \cdot ( \alpha x^*_i(v) ) = \alpha x_i(v).
\end{eqnarray*}
Similarly, $(Mx_i)_{v'}=\alpha x_i(v')$, so that $x_i$ is an eigenvector of $M$ associated with $\alpha \in \{-\lambda,0,\lambda,2\lambda\}$.

Regarding the vector $x_n$, for $u \neq \{v,v'\}$, we have
\begin{eqnarray*}
(Mx_n)_u&=&\sum_{w \in V(G)} m_{uw} x_n(w) = m_{uv} x_n(v)+ m_{uv'} x_n(v')+\sum_{w \notin \{v,v'\}} m_{uw} x_n(w)\\
&=& \frac{m^*_{uv}}{\sqrt{2}} \cdot \frac{1}{\sqrt{2}} +\frac{m^*_{uv}}{\sqrt{2}} \cdot \frac{-1}{\sqrt{2}} + \sum_{w \notin \{v,v'\}} m^*_{uw} \cdot 0=0 = \lambda \cdot x_n(u).
\end{eqnarray*}
We also have
\begin{eqnarray*}
(Mx_n)_v&=&\sum_{w \in V(G)} m_{vw} x_n(w) = m_{vv} x_n(v)+ m_{vv'} x_n(v')+\sum_{w \notin \{v,v'\}} m_{vw} x_n(w)\\
&=& \lambda \cdot \frac{1}{\sqrt{2}} +0+ \sum_{w \notin \{v,v'\}} \frac{m^*_{vw}}{\sqrt{2}} \cdot 0=\lambda x_n(v),\\
(Mx_n)_{v'}&=&\sum_{w \in V(G)} m_{v'w} x_n(w) = m_{v'v'} x_n(v')+ m_{v'v} x_n(v)+\sum_{w \notin \{v,v'\}} m_{vw} x_n(w)\\
&=& \lambda \cdot \frac{-1}{\sqrt{2}} +0+ \sum_{w \notin \{v,v'\}} \frac{m^*_{vw}}{\sqrt{2}} \cdot 0=\lambda x_n(v').
\end{eqnarray*}
Thus $x_n$ is an eigenvector of $M$ associated with the eigenvalue $\lambda$. 

As a consequence, $x_1,\ldots,x_{n}$ are indeed a basis of $\mathbb{R}^n$ consisting of eigenvalues of $M$. Any eigenvalue of $M$ is equal to an eigenvalue of $M^*$ or to $\lambda$, so it lies in $\{-\lambda,0,\lambda,2\lambda\}$, as required.

\vspace{5pt}

\noindent \textbf{Case 3:} $v$ and $v'$ are true twins and $m^*_{vv}=0$

\vspace{5pt}

Define $M=(m_{uw})$ as follows:
\begin{equation}\label{case3}
\begin{array}{ll}
m_{uw}=m^*_{uw}, & \textrm{ if }\{u,w\} \cap \{v,v'\}=\emptyset,\\
m_{uv}=m_{vu}=m^*_{uv}/\sqrt{2}, & \textrm{ if }u \notin \{v,v'\},\\
m_{uv'}=m_{v'u}=m^*_{uv}/\sqrt{2}, & \textrm{ if }u \notin \{v,v'\},\\
m_{vv}=m_{v'v'}=\lambda,m_{v'v}=m_{vv'}=-\lambda.&
\end{array}
\end{equation}

Suppose that the eigenvector $x^*_i$, where $i \in [n-1]$, is associated with the eigenvalue $\alpha$ of $M^*$. For $u \notin \{v,v'\}$, we have
\begin{eqnarray*}
(Mx_i)_u&=&\sum_{w \in V(G)} m_{uw} x_i(w) = m_{uv} x_i(v)+ m_{uv'} x_i(v')+\sum_{w \notin \{v,v'\}} m_{uw} x_i(w)\\
&=& \frac{m^*_{uv}}{\sqrt{2}} \cdot \frac{x^*_i(v)}{\sqrt{2}} +\frac{m^*_{uv}}{\sqrt{2}} \cdot \frac{x^*_i(v)}{\sqrt{2}} + \sum_{w \notin \{v,v'\}} m^*_{uw} x^*_i(w)\\
&=& \sum_{w \in V(G^*)} m^*_{uw} x^*_i(w) = \alpha x^*_i(u) = \alpha x_i(u).
\end{eqnarray*}
We also have
\begin{eqnarray*}
(Mx_i)_v&=&\sum_{w \in V(G)} m_{vw} x_i(w) = m_{vv} x_i(v)+ m_{vv'} x_i(v')+\sum_{w \notin \{v,v'\}} m_{vw} x_i(w)\\
&=& \lambda \cdot \frac{x^*_i(v)}{\sqrt{2}}-\lambda  \cdot \frac{x^*_i(v)}{\sqrt{2}} +\sum_{w \notin \{v,v'\}} \frac{m^*_{vw}}{\sqrt{2}} \cdot x^*_i(w)=m^*_{vv} \cdot \frac{x^*_i(v)}{\sqrt{2}}+ \sum_{w \notin \{v,v'\}} \frac{m^*_{vw}}{\sqrt{2}} \cdot x^*_i(w)\\
&=& \frac{1}{\sqrt{2}} \sum_{w \in V(G^*)} m^*_{vw} x^*_i(w) =  \frac{1}{\sqrt{2}} \cdot ( \alpha x^*_i(v) ) = \alpha x_i(v).
\end{eqnarray*}
Similarly, $(Mx_i)_{v'}=\alpha x_i(v')$, so that $x_i$ is an eigenvector of $M$ associated with $\alpha \in \{-\lambda,0,\lambda,2\lambda\}$.

Regarding the vector $x_n$, for $u \neq \{v,v'\}$, we have
\begin{eqnarray*}
(Mx_n)_u&=&\sum_{w \in V(G)} m_{uw} x_n(w) = m_{uv} x_n(v)+ m_{uv'} x_n(v')+\sum_{w \notin \{v,v'\}} m_{uw} x_n(w)\\
&=& \frac{m^*_{uv}}{\sqrt{2}} \cdot \frac{1}{\sqrt{2}} +\frac{m^*_{uv}}{\sqrt{2}} \cdot \frac{-1}{\sqrt{2}} + \sum_{w \notin \{v,v'\}} m^*_{uw} \cdot 0=0 = (2\lambda) \cdot x_n(u).
\end{eqnarray*}
We also have
\begin{eqnarray*}
(Mx_n)_v&=&\sum_{w \in V(G)} m_{vw} x_n(w) = m_{vv} x_n(v)+ m_{vv'} x_n(v')+\sum_{w \notin \{v,v'\}} m_{vw} x_n(w)\\
&=& \lambda \cdot \frac{1}{\sqrt{2}} + (-\lambda) \cdot \frac{-1}{\sqrt{2}}+ \sum_{w \notin \{v,v'\}} \frac{m^*_{vw}}{\sqrt{2}} \cdot 0=2\lambda \cdot \frac{1}{\sqrt{2}}=(2\lambda)\cdot x_n(v),\\
(Mx_n)_{v'}&=&\sum_{w \in V(G)} m_{v'w} x_n(w) = m_{v'v'} x_n(v')+ m_{v'v} x_n(v)+\sum_{w \notin \{v,v'\}} m_{vw} x_n(w)\\
&=& \lambda \cdot \frac{-1}{\sqrt{2}} + (-\lambda) \cdot \frac{1}{\sqrt{2}} + \sum_{w \notin \{v,v'\}} \frac{m^*_{vw}}{\sqrt{2}} \cdot 0=2\lambda \cdot \frac{-1}{\sqrt{2}}=(2\lambda) \cdot x_n(v').
\end{eqnarray*}
Thus $x_n$ is an eigenvector of $M$ associated with the eigenvalue $2\lambda$. 

As a consequence, $x_1,\ldots,x_{n}$ are indeed a basis of $\mathbb{R}^n$ consisting of eigenvalues of $M$. Any eigenvalue of $M$ is equal to an eigenvalue of $M^*$ or to $2\lambda$, so it lies in $\{-\lambda,0,\lambda,2\lambda\}$, as required.

\vspace{5pt}

\noindent \textbf{Case 4:} $v$ and $v'$ are true twins and $m^*_{vv}=\lambda$

\vspace{5pt}

Define $M=(m_{uw})$ as follows:
\begin{equation}\label{case4}
\begin{array}{ll}
m_{uw}=m^*_{uw}, & \textrm{ if }\{u,w\} \cap \{v,v'\}=\emptyset,\\
m_{uv}=m_{vu}=m^*_{uv}/\sqrt{2}, & \textrm{ if }u \notin \{v,v'\},\\
m_{uv'}=m_{v'u}=m^*_{uv}/\sqrt{2}, & \textrm{ if }u \notin \{v,v'\},\\
m_{vv}=m_{v'v'}=0,m_{v'v}=m_{vv'}=\lambda.&
\end{array}
\end{equation}

Suppose that the eigenvector $x^*_i$, where $i \in [n-1]$, is associated with the eigenvalue $\alpha$ of $M^*$. For $u \notin \{v,v'\}$, we have
\begin{eqnarray*}
(Mx_i)_u&=&\sum_{w \in V(G)} m_{uw} x_i(w) = m_{uv} x_i(v)+ m_{uv'} x_i(v')+\sum_{w \notin \{v,v'\}} m_{uw} x_i(w)\\
&=& \frac{m^*_{uv}}{\sqrt{2}} \cdot \frac{x^*_i(v)}{\sqrt{2}} +\frac{m^*_{uv}}{\sqrt{2}} \cdot \frac{x^*_i(v)}{\sqrt{2}} + \sum_{w \notin \{v,v'\}} m^*_{uw} x^*_i(w)\\
&=& \sum_{w \in V(G^*)} m^*_{uw} x^*_i(w) = \alpha x^*_i(u) = \alpha x_i(u).
\end{eqnarray*}
We also have
\begin{eqnarray*}
(Mx_i)_v&=&\sum_{w \in V(G)} m_{vw} x_i(w) = m_{vv} x_i(v)+ m_{vv'} x_i(v')+\sum_{w \notin \{v,v'\}} m_{vw} x_i(w)\\
&=& 0 \cdot \frac{x^*_i(v)}{\sqrt{2}} + \lambda  \cdot \frac{x^*_i(v)}{\sqrt{2}} +\sum_{w \notin \{v,v'\}} \frac{m^*_{vw}}{\sqrt{2}} \cdot x^*_i(w)=m^*_{vv} \cdot \frac{x^*_i(v)}{\sqrt{2}}+ \sum_{w \notin \{v,v'\}} \frac{m^*_{vw}}{\sqrt{2}} \cdot x^*_i(w)\\
&=& \frac{1}{\sqrt{2}} \sum_{w \in V(G^*)} m^*_{vw} x^*_i(w) =  \frac{1}{\sqrt{2}} \cdot ( \alpha x^*_i(v) ) = \alpha x_i(v).
\end{eqnarray*}
Similarly, $(Mx_i)_{v'}=\alpha x_i(v')$, so that $x_i$ is an eigenvector of $M$ associated with $\alpha \in \{-\lambda,0,\lambda,2\lambda\}$.

Regarding the vector $x_n$, for $u \neq \{v,v'\}$, we have
\begin{eqnarray*}
(Mx_n)_u&=&\sum_{w \in V(G)} m_{uw} x_n(w) = m_{uv} x_n(v)+ m_{uv'} x_n(v')+\sum_{w \notin \{v,v'\}} m_{uw} x_n(w)\\
&=& \frac{m^*_{uv}}{\sqrt{2}} \cdot \frac{1}{\sqrt{2}} +\frac{m^*_{uv}}{\sqrt{2}} \cdot \frac{-1}{\sqrt{2}} + \sum_{w \notin \{v,v'\}} m^*_{uw} \cdot 0=0 = (-\lambda) \cdot x_n(u).
\end{eqnarray*}
We also have
\begin{eqnarray*}
(Mx_n)_v&=&\sum_{w \in V(G)} m_{vw} x_n(w) = m_{vv} x_n(v)+ m_{vv'} x_n(v')+\sum_{w \notin \{v,v'\}} m_{vw} x_n(w)\\
&=& 0 \cdot \frac{1}{\sqrt{2}} + (\lambda) \cdot \frac{-1}{\sqrt{2}}+ \sum_{w \notin \{v,v'\}} \frac{m^*_{vw}}{\sqrt{2}} \cdot 0=-\lambda \cdot \frac{1}{\sqrt{2}}=(-\lambda)\cdot x_n(v),\\
(Mx_n)_{v'}&=&\sum_{w \in V(G)} m_{v'w} x_n(w) = m_{v'v'} x_n(v')+ m_{v'v} x_n(v)+\sum_{w \notin \{v,v'\}} m_{vw} x_n(w)\\
&=& 0 \cdot \frac{-1}{\sqrt{2}} + (\lambda) \cdot \frac{1}{\sqrt{2}} + \sum_{w \notin \{v,v'\}} \frac{m^*_{vw}}{\sqrt{2}} \cdot 0=-\lambda \cdot \frac{-1}{\sqrt{2}}=(-\lambda) \cdot x_n(v').
\end{eqnarray*}
Thus $x_n$ is an eigenvector of $M$ associated with the eigenvalue $-\lambda$. 

As a consequence, $x_1,\ldots,x_{n}$ are indeed a basis of $\mathbb{R}^n$ consisting of eigenvalues of $M$. Any eigenvalue of $M$ is equal to an eigenvalue of $M^*$ or to $-\lambda$, so it lies in $\{-\lambda,0,\lambda,2\lambda\}$, as required.

\end{proof}

\bibliographystyle{plain}
\bibliography{mybibliography.bib}

\end{document}